\documentclass[12pt]{amsart}
\usepackage{bbding}
\usepackage[margin=2cm]{geometry}

\usepackage{amsfonts}
\usepackage{mathrsfs}
\usepackage{cite}
\usepackage{tikz}

\newtheorem{theorem}{Theorem}[section]
\newtheorem{lemma}[theorem]{Lemma}

\newcommand{\Int}[1]{\text{Int }#1}

\newcommand{\Span}[1]{\text{Span }#1}

\title{The $L_{p}$ Minkowski problem for polytopes for $0<p<1$}
\author{guangxian zhu}
\address{DEPARTMENT OF MATHEMATICS, POLYTECHNIC SCHOOL OF ENGINEERING, NEW YORK UNIVERSITY, SIX METROTECH, BROOKLYN, NY 11201, UNITED STATES}
\email{gz342@nyu.edu}

\begin{document}
\maketitle

\begin{abstract}
Necessary and sufficient conditions are given for the existence of solutions to the discrete $L_{p}$ Minkowski problem for the critical case where $0<p<1$.
\end{abstract}

\section{Introduction}

The setting for this paper is $n$-dimensional Euclidean space $\mathbb{R}^{n}$. A \emph{convex body} in $\mathbb{R}^{n}$ is a compact convex set that
has non-empty interior. If $K$ is a convex body in $\mathbb{R}^{n}$, then the
\emph{surface area measure}, $S_{K}$, of $K$ is a Borel measure on the unit sphere, $S^{n-1}$, defined for a Borel $\omega\subset S^{n-1}$, by
$$
S_{K}(\omega)=\int_{x\in\nu_{K}^{-1}(\omega)}d\mathcal{H}^{n-1}(x),
$$
where $\nu_{K}:\partial'K\rightarrow S^{n-1}$ is the Gauss map of $K$, defined on $\partial'K$, the set of points of $\partial K$ that have a
unique outer unit normal, and $\mathcal{H}^{n-1}$ is ($n-1$)-dimensional Hausdorff measure.

The Minkowski problem is one of the cornerstones of the classical Brunn-Minkowski theory: What are necessary and sufficient conditions on a finite Borel measure $\mu$ on $S^{n-1}$ so that it is the surface area measure of a convex body in $\mathbb{R}^{n}$?

More than a century ago, Minkowski himself solved his problem for the case where the given measure is discrete \cite{MIN}. The complete solution to this problem for
arbitrary measures was given by Aleksandrov, and Fenchel and Jessen (see, e.g., \cite{SCH}): If $\mu$ is not concentrated on a great subsphere of $S^{n-1}$, then $\mu$ is the surface area measure of a convex body if and only if
$$
\int_{S^{n-1}}ud\mu=0.
$$

In \cite{LUT},  Lutwak showed that there is an $L_{p}$ analogue of the surface area measure and posed the associated $L_{p}$ Minkowski problem which has the classical Minkowski problem as an important case. If $p\in\mathbb{R}$ and $K$ is a convex body in $\mathbb{R}^{n}$ that contains the origin in its interior, then the $L_{p}$ surface area measure,
$S_{p}(K,\cdot)$, of $K$ is a Borel measure on $S^{n-1}$ defined for a Borel $\omega\subset S^{n-1}$, by
$$
S_{p}(K,\omega)=\int_{x\in\nu_{K}^{-1}(\omega)}(x\cdot\nu_{K}(x))^{1-p}d\mathcal{H}^{n-1}(x).
$$
Obviously, $S_{1}(K,\cdot)$ is the classical surface area measure of $K$. In recent years, the $L_{p}$ surface area measure appeared in, e.g., \cite{ADA, BG, CG, GM, H1, HP1, HP2, HENK, LU1, LU2, LR, LYZ1, LYZ2, LYZ3, LYZ6, LZ, NAO, NR, PAO, PW, ST3}.

Today, the $L_{p}$ Minkowski problem is one of the central problems in convex geometric analysis. It can be stated in the following way:\\ \\
\textbf{$L_{p}$ Minkowski problem:} For fixed $p$, what are necessary and sufficient conditions on a finite Borel measure $\mu$ on $S^{n-1}$ so that $\mu$ is the $L_{p}$ surface area measure of a convex body in $\mathbb{R}^{n}$?

When $\mu$ has a density $f$, with respect to spherical Lebesgue measure, the $L_{p}$ Minkowski problem involves establishing existence for the
Monge-Amp\`{e}re type equation:
\begin{equation}\tag{1.1}\label{Equation 1.1}
h^{1-p}\det(h_{ij}+h\delta_{ij})=f,
\end{equation}
where $h_{ij}$ is the covariant derivative of $h$ with respect to an orthonormal frame on $S^{n-1}$ and $\delta_{ij}$ is the Kronecker delta.

Obviously, the $L_{1}$ Minkowski problem is the classical Minkowski problem. Establishing existence and uniqueness for the solution of the classical Minkowski problem was done by Aleksandrov, and Fenchel and Jessen (see, e.g., \cite{SCH}). When $p\neq1$, the $L_{p}$ Minkowski problem has been studied by, e.g., Lutwak \cite{LUT}, Lutwak and Oliker \cite{LO}, Guan and Lin \cite{GL}, Chou and Wang \cite{CW}, Hug, et al. \cite{HLYZ2}, B\"{o}r\"{o}czky, et al. \cite{BLYZ}. Additional references regarding the $L_{p}$ Minkowski problem and Minkowski-type problems can be found in \cite{BLYZ, WC, CW, GG, GL, GM, H1, HL1, HLYZ2, HLYZ1, HMS, JI, KL, LWA, LUT, LO, LYZ5, ST1, ST2}.

The solutions to the Minkowski problem and the $L_{p}$ Minkowski problem connect with some important flows (see, e.g., \cite{AN1, AN2, KSC, GH, HUS}), and have important applications to Sobolev-type inequalities, see, e.g., Zhang \cite{Z}, Lutwak, et al. \cite{LYZ4}, Ciachi, et al. \cite{CLYZ}, Haberl and Schuster \cite{HS1, HS2, HSX}, and Wang \cite{TUO}.

Most previous work on the $L_{p}$ Minkowski problem was limited to the case where $p>1$. The reason that uniqueness of solutions to the $L_{p}$ Minkowski problem for $p>1$ can be shown is the availability of mixed volume inequalities established by Lutwak \cite{LUT}. One reason that the $L_{p}$ Minkowski problem becomes challenging when $p<1$ is because little is known about the mixed volume inequalities when $p<1$ (see, e.g., \cite{BLYZ2}). In $\mathbb{R}^{n}$, necessary and sufficient conditions for the existence of the solution of the even $L_{p}$ Minkowski problem for the case of $0<p<1$ was given by Haberl, et al. \cite{HLYZ1}. Necessary and sufficient conditions for the existence of solutions to the even $L_{0}$ Minkowski problem (also called the logarithmic Minkowski problem) was recently established by B\"{o}r\"{o}czky, et al. \cite{BLYZ}. Without the assumption that the measure is even, existence of solution of the PDE (\ref{Equation 1.1}) for the case where $p>-n$ were given by Chou and Wang \cite{CW}. In \cite{Z2, Z3}, the author established necessary and sufficient conditions for the existence of the solution of the $L_{p}$ Minkowski problem for the case where $p=0$ and $p=-n$, and $\mu$ is a discrete measure whose support-vectors (i.e., vectors in the support of the measure) are in general position.

One reason that the Minkowski and the $L_{p}$ Minkowski problem for polytopes are important is because the Minkowski problem and the $L_{p}$ Minkowski problem (for $p>1$) for measures can be solved by an approximation argument by first solving the polytopal case (see, e.g., Hug, et al. \cite{HLYZ2} and Schneider \cite{SCH}, pp. 392-393).

A \emph{polytope} in $\mathbb{R}^{n}$ is the convex hull of a finite set of points in $\mathbb{R}^{n}$ provided that it has positive $n$-dimensional volume.
The convex hull of a subset of these points is called a \emph{facet} of the polytope if it lies entirely on the boundary of the polytope and has
positive ($n-1$)-dimensional volume. If a polytope $P$ contains the origin in its interior with $N$ facets whose outer unit normals are
$u_{1},...,u_{N}$, and such that if the facet with outer unit normal $u_{k}$ has area $a_{k}$ and distance from the origin $h_{k}$ for
all $k\in\{1,...,N\}$. Then,
$$
S_{p}(P,\cdot)=\sum_{k=1}^{N}h_{k}^{1-p}a_{k}\delta_{u_{k}}(\cdot).
$$
where $\delta_{u_{k}}$ denotes the delta measure that is concentrated at the point $u_{k}$.\\

It is the aim of this paper to establish:\\ \\
\textbf{Theorem.} \emph{If $p\in(0,1)$, and $\mu$ is a discrete measure on the unit sphere, then $\mu$ is the $L_{p}$ surface area measure of a polytope if and only if the support of $\mu$ is not concentrated on a closed hemisphere.}\\

This paper is organized as follows. In Section 2, we recall some basic facts about convex bodies. In Section 3,
we study an extremal problem related to the $L_{p}$ Minkowski problem. In Section 4, we prove the main theorem of this paper.

For the case where $p>1$ with $p\neq n$, necessary and sufficient conditions for the existence of solutions to the discrete $L_{p}$ Minkowski problem were established by Hug, et al. \cite{HLYZ2}. In Section 5, we give a new proof of this condition. The proof presented in this paper includes a new approach to the classical Minkowski problem.

\section{Preliminaries}

In this section, we collect some basic definitions and facts about convex bodies. For general references regarding convex bodies see, e.g., \cite{GAR1, PM, GRU, SCH, THO}.

The sets of this paper are subsets of the $n$-dimensional Euclidean space $\mathbb{R}^{n}$. For $x, y\in\mathbb{R}^{n}$, we write $x\cdot y$
for the standard inner product of $x$ and $y$, $|x|$ for the Euclidean norm of $x$ and $B^{n}$ for the unit ball of $\mathbb{R}^{n}$.

For $K_{1}, K_{2}\subset\mathbb{R}^{n}$ and $c_{1}, c_{2}\geq0$, the Minkowski combination, $c_{1}K_{1}+c_{2}K_{2}$, is defined by
$$
c_{1}K_{1}+c_{2}K_{2}=\{c_{1}x_{1}+c_{2}x_{2}: x_{1}\in K_{1}, x_{2}\in K_{2}\}.
$$
The \emph{support function}
$h_{K}: \mathbb{R}^{n}\rightarrow\mathbb{R}$ of a compact convex set
$K$ is defined, for $x\in\mathbb{R}^{n}$, by
$$
h(K,x)=\max\{x\cdot y: y\in K\}.
$$
Obviously, for $c\geq0$ and $x\in\mathbb{R}^{n}$,
$$
h(cK,x)=h(K,cx)=ch(K,x).
$$

The \emph{Hausdorff distance} between two compact sets $K, L$ in
$\mathbb{R}^{n}$ is defined by
$$
\delta(K,L)=\inf\{t\geq0: K\subset L+tB^{n}, L\subset K+tB^{n}\}.
$$
It is easily shown that the Hausdorff distance between two convex bodies, $K$ and $L$, is
$$
\delta(K,L)=\max_{u\in S^{n-1}}|h(K,u)-h(L,u)|.
$$

For a convex body $K$ in $\mathbb{R}^{n}$, and $u\in S^{n-1}$,
the \emph{support hyperplane} $H(K,u)$ in direction $u$ is defined by
$$
H(K,u)=\{x\in\mathbb{R}^{n}:x\cdot u=h(K,u)\},
$$
the \emph{half-space} $H^{-}(K,u)$ in direction $u$ is defined by
$$
H^{-}(K,u)=\{x\in\mathbb{R}^{n}: x\cdot u\leq h(K,u)\},
$$
and the \emph{support set} $F(K,u)$ in direction $u$ is defined by
$$
F(K,u)=K\cap H(K,u).
$$

For a compact $K\in\mathbb{R}^{n}$, the diameter of $K$ is defined by
$$
d(K)=\max\{|x-y|: x,y\in K\}.
$$

Let $\mathcal{P}$ be the set of polytopes in
$\mathbb{R}^{n}$. If the unit vectors $u_{1},...,u_{N}$ ($N\geq n+1$) are not concentrated on
a closed hemisphere, let $\mathcal{P}(u_{1},...,u_{N})$ be the
subset of $\mathcal{P}$ such that a polytope
$P\in\mathcal{P}(u_{1},...,u_{N})$ if
$$
P=\bigcap_{k=1}^{N}H^{-}(P,u_{k}).
$$
Obviously, if $P\in\mathcal{P}(u_{1},...,u_{N})$, then $P$ has at
most $N$ facets, and the outer unit normals of $P$ are a subset of
$\{u_{1},...,u_{N}\}$. Let $\mathcal{P}_{N}(u_{1},...,u_{N})$
be the subset of $\mathcal{P}(u_{1},...,u_{N})$ such that a
polytope
$P\in\mathcal{P}_{N}(u_{1},...,u_{N})$ if,
$P\in\mathcal{P}(u_{1},...,u_{N})$ and $P$ has exactly $N$
facets.

\section{An extreme problem related to the $L_{p}$ Minkowski problem}

Suppose $p\in(0,1)$, $\alpha_{1},...,\alpha_{N}\in\mathbb{R}^{+}$, the unit vectors $u_{1},...,u_{N}$ ($N\geq n+1$) are not concentrated on a closed hemisphere, and $P\in\mathcal{P}(u_{1},...,u_{N})$. Define the function, $\Phi_{P}: P\rightarrow\mathbb{R}$, by
$$
\Phi_{P}(\xi)=\sum_{k=1}^{N}\alpha_{k}(h(P,u_{k})-\xi\cdot u_{k})^{p}.
$$

In this section, we study the extremal problem
\begin{equation}\tag{3.0}\label{Equation 3.0}
\inf\{\sup_{\xi\in Q}\Phi_{Q}(\xi): Q\in\mathcal{P}(u_{1},...,u_{N})\textmd{ and }V(Q)=1\}.
\end{equation}
We will prove that $\Phi_{P}(\xi)$ is strictly concave on $P$ and that there exists a unique $\xi_{p}(P)\in\Int(P)$ such that
$$
\Phi_{P}(\xi_{p}(P))=\sup_{\xi\in P}\Phi_{P}(\xi).
$$
Moreover, we will prove that there exists a polytope with $u_{1},...,u_{N}$ as its outer unit normals, and the polytope solving problem (\ref{Equation 3.0}).

We first prove the concavity of $\Phi_{P}(\xi)$.
\begin{lemma}\label{Lemma 3.1}
 If $0<p<1$, $\alpha_{1},...,\alpha_{N}\in\mathbb{R}^{+}$, the unit vectors $u_{1},...,u_{N}$ ($N>n+1$) are not concentrated on a closed hemisphere and $P\in\mathcal{P}(u_{1},...,u_{N})$, then $\Phi_{P}(\xi)$ is strictly concave on $P$.
\end{lemma}
\begin{proof}
For $0<p<1$, $t^{p}$ is strictly concave on $[0,+\infty)$. Thus, for $0<\lambda<1$
and $\xi_{1}, \xi_{2}\in P$,
\begin{equation*}
\begin{split}
\lambda\Phi_{P}(\xi_{1})+(1-\lambda)\Phi_{P}(\xi_{2})&=\lambda\sum_{k=1}^{N}\alpha_{k}(h(P,u_{k})-\xi_{1}\cdot
u_{k})^{p}+(1-\lambda)\sum_{k=1}^{N}\alpha_{k}(h(P,u_{k})-\xi_{2}\cdot
u_{k})^{p}\\
&=\sum_{k=1}^{N}\alpha_{k}\left[\lambda(h(P,u_{k})-\xi_{1}\cdot
u_{k})^{p}+(1-\lambda)(h(P,u_{k})-\xi_{2}\cdot u_{k})^{p}\right]\\
&\leq\sum_{k=1}^{N}\alpha_{k}\left[h(P,u_{k})-(\lambda\xi_{1}+(1-\lambda)\xi_{2})\cdot
u_{k}\right]^{p}\\
&=\Phi_{P}(\lambda\xi_{1}+(1-\lambda)\xi_{2}),
\end{split}
\end{equation*}
with equality if and only if $\xi_{1}\cdot u_{k}=\xi_{2}\cdot u_{k}$
for all $k=1,...,N$. Since $u_{1},...,u_{N}$ are not concentrated on a closed hemisphere, $\mathbb{R}^{n}=\Span\{u_{1},...,u_{N}\}$. Thus, $\xi_{1}=\xi_{2}$. Therefore, $\Phi_{P}$ is strictly concave
on $P$.
\end{proof}

The following lemma is needed.
\begin{lemma}\label{Lemma 3.2}
If $0<p<1$, $\alpha_{1},...,\alpha_{N}\in\mathbb{R}^{+}$, the unit vectors $u_{1},...,u_{N}$ ($N>n+1$) are not concentrated on a closed hemisphere and $P\in\mathcal{P}(u_{1},...,u_{N})$, then there exists a unique $\xi_{p}(P)\in\Int(P)$ such that
$$
\Phi_{P}(\xi_{p}(P))=\max_{\xi\in P}\Phi_{P}(\xi),
$$
where $\Phi_{P}(\xi)=\sum_{k=1}^{N}\alpha_{k}(h(P,u_{k})-\xi\cdot u_{k})^{p}$.
\end{lemma}
\begin{proof}
From Lemma \ref{Lemma 3.1}, for $0<p<1$, $\Phi_{P}(\xi)$ is strictly concave on $P$. From this and the fact that $P$ is a compact convex set, we have, there
exists a unique $\xi_{p}(P)\in P$ such that
$$
\Phi_{P}(\xi_{p}(P))=\max_{\xi\in P}\Phi_{P}(\xi).
$$
We next prove that $\xi_{p}(P)\in\Int(P)$. Otherwise, suppose $\xi_{p}(P)\in\partial P$ with
$$
h(P,u_{k})-\xi_{p}(P)\cdot u_{k}=0
$$
for $k\in\{i_{1},...,i_{m}\}$ and
$$
h(P,u_{k})-\xi_{p}(P)\cdot u_{k}>0
$$
for $k\in\{1,...,N\}\backslash\{i_{1},...,i_{m}\}$, where $1\leq
i_{1}<...<i_{m}\leq N$ and $1\leq m\leq N-1$. Choose $x_{0}\in\Int(P)$, let
$$
u_{0}=\frac{x_{0}-\xi_{p}(P)}{|x_{0}-\xi_{p}(P)|},
$$
and let
\begin{equation}\tag{3.1}\label{Equation 3.1}
\begin{split}
\left[h(P,u_{k})-(\xi_{p}(P)+\delta u_{0})\cdot u_{k}\right]-\left[h(P,u_{k})-\xi_{p}(P)\cdot u_{k}\right]&=(-u_{0}\cdot u_{k})\delta\\
&=c_{k}\delta,
\end{split}
\end{equation}
where $c_{k}=-u_{0}\cdot u_{k}$. Since $h(P,u_{k})-\xi_{p}(P)\cdot u_{k}=0$ for $k\in\{i_{1},...,i_{m}\}$ and $x_{0}$ is an interior point of $P$, $c_{k}=-u_{0}\cdot u_{k}>0$ for $k\in\{i_{1},...,i_{m}\}$. Let
\begin{equation}\tag{3.2}\label{Equation 3.2}
c_{0}=\min\big\{h(P,u_{k})-\xi_{p}(P)\cdot u_{k}: k\in\{1,...,N\}\backslash\{i_{1},...,i_{m}\}\big\}>0,
\end{equation}
and choose $\delta>0$ small enough so that $\xi_{p}(P)+\delta u_{0}\in\Int(P)$ and
\begin{equation}\tag{3.3}\label{Equation 3.3}
\min\big\{h(P,u_{k})-\big(\xi_{p}(P)+\delta u_{0}\big)\cdot u_{k}: k\in\{1,...,N\}\backslash\{i_{1},...,i_{m}\}\big\}>\frac{c_{0}}{2}.
\end{equation}

Obviously, for $0<p<1$ and $x_{0}, x_{0}+\Delta x\in(\frac{c_{0}}{2},+\infty)$,
$$
|(x_{0}+\Delta x)^{p}-x_{0}^{p}|<p(\frac{c_{0}}{2})^{p-1}|\Delta x|.
$$
From this, the fact that $h(P,u_{k})=\xi_{p}(P)\cdot u_{k}$ for $k\in\{i_{1},...,i_{m}\}$, the fact $c_{k}>0$ for $k\in\{i_{1},...,i_{m}\},$ and Equations (\ref{Equation 3.1}), (\ref{Equation 3.2}) and (\ref{Equation 3.3}), we have
\begin{equation*}
\begin{split}
\Phi_{P}(\xi_{p}(P)+\delta u_{0})-\Phi_{P}(\xi_{p}(P))&=\sum_{k=1}^{N}\alpha_{k}\bigg[\big(h(P,u_{k})-(\xi_{p}(P)+\delta u_{0})\cdot
u_{k}\big)^{p}-\big(h(P,u_{k})-\xi_{p}(P)\cdot u_{k}\big)^{p}\bigg]\\
&\geq\sum_{k\in\{i_{1},...,i_{m}\}}\alpha_{k}(c_{k}\delta)^{p}-\sum_{k\in\{1,...,N\}\backslash\{i_{1},...,i_{m}\}}\alpha_{k}\bigg|\big(h(P,u_{k})-\xi_{p}(P)\cdot
u_{k}\\
&\qquad+c_{k}\delta\big)^{p}-\big(h(P,u_{k})-\xi_{p}(P)\cdot u_{k}\big)^{p}\bigg|\\
&\geq\bigg(\sum_{k\in\{i_{1},...,i_{m}\}}\alpha_{k}c_{k}^{p}\bigg)\delta^{p}-\sum_{k\in\{1,...,N\}\backslash\{i_{1},...,i_{m}\}}\alpha_{k}p(\frac{c_{0}}{2})^{p-1}|c_{k}\delta|\\
&=\bigg(\sum_{k\in\{i_{1},...,i_{m}\}}\alpha_{k}c_{k}^{p}-\sum_{k\in\{1,...,N\}\backslash\{i_{1},...,i_{m}\}}\alpha_{k}p(\frac{c_{0}}{2})^{p-1}|c_{k}|\delta^{1-p}\bigg)\delta^{p}.
\end{split}
\end{equation*}
Thus, there exists a small enough $\delta_{0}>0$ such that $\xi_{p}(P)+\delta_{0}u_{0}\in\Int(P)$ and
$$
\Phi_{P}(\xi_{p}(P)+\delta_{0}u_{0})>\Phi_{P}(\xi_{p}(P)).
$$
This contradicts the definition of $\xi_{p}(P)$. Therefore, $\xi_{p}(P)\in\Int(P)$.
\end{proof}

By definition, for $\lambda>0$, $p\in(0,1)$ and $P\in\mathcal{P}(u_{1},...,u_{N})$,
\begin{equation}\tag{3.4}\label{Equation 3.4}
\xi_{p}(\lambda P)=\lambda\xi_{p}(P).
\end{equation}

Obviously, if $P_{i}\in\mathcal{P}(u_{1},...,u_{N})$ and $P_{i}$ converges to a polytope $P$, then $P\in\mathcal{P}(u_{1},...,u_{N})$.
\begin{lemma}\label{Lemma 3.3}
If $p\in(0,1)$, $\alpha_{1},...,\alpha_{N}$ are positive, the unit vectors $u_{1},...,u_{N}$ ($N\geq n+1$) are not concentrated on a closed hemisphere, $P_{i}\in\mathcal{P}(u_{1},...,u_{N})$, and $P_{i}$ converges to a polytope $P$. Then,
$\lim_{i\rightarrow\infty}\xi_{p}(P_{i})=\xi_{p}(P)$ and
$$
\lim_{i\rightarrow\infty}\Phi_{P_{i}}(\xi_{p}(P_{i}))=\Phi_{P}(\xi_{p}(P)),
$$
where $\Phi_{P}(\xi)=\sum_{k=1}^{N}\alpha_{k}(h(P,u_{k})-\xi\cdot u_{k})^{p}$.
\end{lemma}
\begin{proof}
Since $P_{i}\rightarrow P$ and $\xi_{p}(P_{i})\in\Int(P_{i})$, $\xi_{p}(P_{i})$ is bounded. Suppose
$\xi_{p}(P_{i})$ does not converge to $\xi_{p}(P)$, then there exists a subsequence $P_{i_{j}}$ of $P_{i}$ such that $P_{i_{j}}$ converges to $P$,
$\xi_{p}(P_{i_{j}})\rightarrow\xi_{0}$ but $\xi_{0}\neq\xi_{p}(P)$. Obviously, $\xi_{0}\in P$ and
\begin{equation*}
\begin{split}
\lim_{j\rightarrow\infty}\Phi_{P_{i_{j}}}(\xi_{p}(P_{i_{j}}))&=\Phi_{P}(\xi_{0})\\
&<\Phi_{P}(\xi_{p}(P))\\
&=\lim_{j\rightarrow\infty}\Phi_{P_{i_{j}}}(\xi_{p}(P)).
\end{split}
\end{equation*}
This contradicts the fact that
$$
\Phi_{P_{i_{j}}}(\xi_{p}(P_{i_{j}}))\geq\Phi_{P_{i_{j}}}(\xi_{p}(P)).
$$
Therefore, $\lim_{i\rightarrow\infty}\xi_{p}(P_{i})=\xi_{p}(P)$ and thus,
$$
\lim_{i\rightarrow\infty}\Phi_{P_{i}}(\xi_{p}(P_{i}))=\Phi_{P}(\xi_{p}(P)).
$$
\end{proof}

The following lemma is useful in the proving of the compactness of problem (\ref{Equation 3.0}).
\begin{lemma}\label{Lemma 3.4}
 Suppose $p>0$, $\alpha_{1},...,\alpha_{N}$ are positive, and the unit vectors $u_{1},...,u_{N}$ ($N\geq n+1$) are not concentrated on a hemisphere. If $P_{k}\in\mathcal{P}(u_{1},...,u_{N})$, $o\in P_{k}$, and $R(P_{k})$ is not bounded, then
$$
\sum_{i=1}^{N}\alpha_{i}h(P_{k},u_{i})^{p}
$$
is not bounded.
\end{lemma}
\begin{proof}
Without loss of generality, we can assume
$$
\lim_{k\rightarrow\infty}R(P_{k})=\infty.
$$

Let
$$
f(u)=\sum_{i=1}^{N}\alpha_{i}|u\cdot u_{i}|^{p},
$$
where $u\in S^{n-1}$.

Since $u_{1},...,u_{N}$ are not contained in a closed hemisphere, $\mathbb{R}^{n}=\Span\{u_{1},...,u_{N}\}$. Thus, $f(u)>0$ for all $u\in S^{n-1}$. On the other hand, $f(u)$ is continuous on $S^{n-1}$. Thus, there exists a constant $a_{0}>0$ such that
$$
\sum_{i=1}^{N}\alpha_{i}|u\cdot u_{i}|^{p}\geq a_{0}
$$
for all $u\in S^{n-1}$.

Choose $u_{k}\in S^{n-1}$ such that $R(P_{k})u_{k}\in P_{k}$. Since $o\in P_{k}$,
\begin{equation*}
\begin{split}
\sum_{i=1}^{N}\alpha_{i}h(P_{k},u_{i})^{p}&\geq\sum_{i=1}^{N}\alpha_{i}|R(P_{k})u_{k}\cdot u_{i}|^{p}\\
&=R(P_{k})^{p}\big(\sum_{i=1}^{N}\alpha_{i}|u_{k}\cdot u_{i}|^{p}\big)\\
&\geq a_{0}R(P_{k})^{p}\rightarrow+\infty.
\end{split}
\end{equation*}
\end{proof}

The following lemma will be needed.
\begin{lemma}\label{Lemma 3.5}
If $P$ is a polytope in $\mathbb{R}^{n}$ and $v_{0}\in S^{n-1}$ with $V_{n-1}(F(P,v_{0}))=0$, then there exists a $\delta_{0}>0$ such that for $0\leq\delta<\delta_{0}$
$$
V(P\cap\{x: x\cdot v_{0}\geq h(P,v_{0})-\delta\})=c_{n}\delta^{n}+...+c_{2}\delta^{2},
$$
where $c_{n},...,c_{2}$ are constants that depend on $P$ and $v_{0}$.
\end{lemma}
\begin{proof}
It is known (e.g., \cite{GGR}, Proposition 3.1) that
$$
g(\delta)=V_{n-1}(P\cap\{x: x\cdot v_{0}= h(P,v_{0})-\delta\})
$$
is a piecewise polynomial function of degree at most $n-1$. By conditions, $g(0)=0$. Thus, there exists a $\delta_{0}>0$ and $c_{n-1}',...,c_{1}'$ (depend on $P$ and $v_{0}$) such that when $0\leq\delta<\delta_{0}$
$$
g(\delta)=c_{n-1}'\delta^{n-1}+...+c_{1}'\delta.
$$
Therefore, when $0\leq\delta<\delta_{0}$,
\begin{equation*}
\begin{split}
V(P\cap\{x: x\cdot v_{0}\geq h(P,v_{0})-\delta\})&=\int_{0}^{\delta}g(t)dt\\
&=c_{n}\delta^{n}+...+c_{2}\delta^{2},
\end{split}
\end{equation*}
where $c_{n}=c_{n-1}'/n,...,c_{2}=c_{1}'/2$ are constants that depend on $P$ and $v_{0}$.
\end{proof}

We next solve problem (\ref{Equation 3.0}).
\begin{lemma}\label{Lemma 3.6}
If $0<p<1$, $\alpha_{1},...,\alpha_{N}$ are positive, and the unit vectors $u_{1},...,u_{N}$ ($N\geq n+1$) are not concentrated on a hemisphere,
then there exists a $P\in\mathcal{P}_{N}(u_{1},...,u_{N})$ such that $\xi_{p}(P)=o$, $V(P)=1$, and
$$
\Phi_{P}(o)=\inf\{\max_{\xi\in Q}\Phi_{Q}(\xi): Q\in\mathcal{P}(u_{1},...,u_{N})\textmd{ and }V(Q)=1\},
$$
where $\Phi_{Q}(\xi)=\sum_{k=1}^{N}\alpha_{k}(h(Q,u_{k})-\xi\cdot u_{k})^{p}$.
\end{lemma}
\begin{proof}
Obviously, for $P, Q\in\mathcal{P}(u_{1},...,u_{N})$, if there exists a $x\in\mathbb{R}^{n}$ such that $P=Q+x$, then
$$
\Phi_{P}(\xi_{p}(P))=\Phi_{Q}(\xi_{p}(Q)).
$$
Thus, we can choose a sequence $P_{i}\in\mathcal{P}(u_{1},...,u_{N})$ with $\xi_{p}(P_{i})=o$ and $V(P_{i})=1$ such that $\Phi_{P_{i}}(o)$ converges
to
$$
\inf\{\max_{\xi\in Q}\Phi_{Q}(\xi): Q\in\mathcal{P}(u_{1},...,u_{N})\textmd{ and }V(Q)=1\}.
$$

Choose a fixed $P_{0}\in\mathcal{P}(u_{1},...,u_{N})$ with $V(P_{0})=1$, then
$$
\inf\{\max_{\xi\in Q}\Phi_{Q}(\xi): Q\in\mathcal{P}(u_{1},...,u_{N})\textmd{ and }V(Q)=1\}\leq\Phi_{P_{0}}(\xi_{p}(P_{0})).
$$

We claim that $P_{i}$ is bounded. Otherwise, from Lemma \ref{Lemma 3.4}, $\Phi_{P_{i}}(\xi_{p}(P_{i}))$ converges to $+\infty$. This contradicts the previous inequality. Therefore, $P_{i}$ is bounded.

From Lemma \ref{Lemma 3.3} and the Blaschke selection theorem, there exists a subsequence of $P_{i}$ that converges to a
polytope $P$ such that $P\in\mathcal{P}(u_{1},...,u_{N})$, $V(P)=1$, $\xi(P)=o$ and
\begin{equation}\tag{3.6}\label{Equation 3.6}
\Phi_{P}(o)=\inf\{\max_{\xi\in Q}\Phi_{Q}(\xi): Q\in\mathcal{P}(u_{1},...,u_{N})\textmd{ and }V(Q)=1\}.
\end{equation}

We next prove that $F(P,u_{i})$ are facets for all $i=1,...,N$. Otherwise, there exists a $i_{0}\in\{1,...,N\}$ such that
$$
F(P,u_{i_{0}})
$$
is not a facet of $P$.

Choose $\delta>0$ small enough so that the polytope
$$
P_{\delta}=P\cap\{x: x\cdot u_{i_{0}}\leq h(P,u_{i_{0}})-\delta\}\in\mathcal{P}(u_{1},...,u_{N}).
$$
and (by Lemma \ref{Lemma 3.5})
$$
V(P_{\delta})=1-(c_{n}\delta^{n}+...+c_{2}\delta^{2}),
$$
where $c_{n},...,c_{2}$ are constants that depend on $P$ and direction $u_{i_{0}}$.

From Lemma \ref{Lemma 3.3}, for any $\delta_{i}\rightarrow0$ it always true that $\xi_{p}(P_{\delta_{i}})\rightarrow o$. We have,
$$
\lim_{\delta\rightarrow0}\xi_{p}(P_{\delta})=o.
$$
Let $\delta$ be small enough so that $h(P,u_{k})>\xi_{p}(P_{\delta})\cdot u_{k}+\delta$ for all $k\in\{1,...,N\}$, and let
$$
\lambda=V(P_{\delta})^{-\frac{1}{n}}=(1-(c_{n}\delta^{n}+...+c_{2}\delta^{2}))^{-\frac{1}{n}}.
$$
From this and Equation (\ref{Equation 3.4}), we have
\begin{equation}\tag{3.7}\label{Equation 3.7}
\begin{split}
\Phi_{\lambda P_{\delta}}(\xi_{p}(\lambda P_{\delta}))&=\sum_{k=1}^{N}\alpha_{k}\big(h(\lambda P_{\delta},u_{k})-\xi_{p}(\lambda
P_{\delta})\cdot
u_{k}\big)^{p}\\
&=\lambda^{p}\sum_{k=1}^{N}\alpha_{k}\big(h(P_{\delta},u_{k})-\xi_{p}(P_{\delta})\cdot u_{k}\big)^{p}\\
&=\lambda^{p}\sum_{k=1}^{N}\alpha_{k}\big(h(P,u_{k})-\xi_{p}(P_{\delta})\cdot
u_{k}\big)^{p}-\alpha_{i_{0}}\lambda^{p}\big(h(P,u_{i_{0}})-\xi_{p}(P_{\delta})\cdot
u_{i_{0}}\big)^{p}\\
&\qquad+\alpha_{i_{0}}\lambda^{p}\big(h(P,u_{i_{0}})-\xi_{p}(P_{\delta})\cdot u_{i_{0}}-\delta\big)^{p}\\
&=\sum_{k=1}^{N}\alpha_{k}\big(h(P,u_{k})-\xi_{p}(P_{\delta})\cdot
u_{k}\big)^{p}+(\lambda^{p}-1)\sum_{k=1}^{N}\alpha_{k}\big(h(P,u_{k})-\xi_{p}(P_{\delta})\cdot u_{k}\big)^{p}\\
&\qquad+\alpha_{i_{0}}\lambda^{p}\bigg[\big(h(P,u_{i_{0}})-\xi_{p}(P_{\delta})\cdot
u_{i_{0}}-\delta\big)^{p}-\big(h(P,u_{i_{0}})-\xi_{p}(P_{\delta})\cdot u_{i_{0}}\big)^{p}\bigg]\\
&=\Phi_{P}(\xi_{p}(P_{\delta}))+B(\delta),
\end{split}
\end{equation}
where
\begin{equation*}
\begin{split}
B(\delta)&=(\lambda^{p}-1)\left(\sum_{k=1}^{N}\alpha_{k}\big(h(P,u_{k})-\xi_{p}(P_{\delta})\cdot u_{k}\big)^{p}\right)\\
&\qquad+\alpha_{i_{0}}\lambda^{p}\bigg[\big(h(P,u_{i_{0}})-\xi_{p}(P_{\delta})\cdot
u_{i_{0}}-\delta\big)^{p}-\big(h(P,u_{i_{0}})-\xi_{p}(P_{\delta})\cdot
u_{i_{0}}\big)^{p}\bigg]\\
&=\left[(1-(c_{n}\delta^{n}+...+c_{2}\delta^{2}))^{-\frac{p}{n}}-1\right]\left(\sum_{k=1}^{N}\alpha_{k}(h(P,u_{k})-\xi_{p}(P_{\delta})\cdot u_{k})^{p}\right)\\
&\qquad+\alpha_{i_{0}}\lambda^{p}\bigg[\big(h(P,u_{i_{0}})-\xi_{p}(P_{\delta})\cdot
u_{i_{0}}-\delta\big)^{p}-\big(h(P,u_{i_{0}})-\xi_{p}(P_{\delta})\cdot u_{i_{0}}\big)^{p}\bigg].
\end{split}
\end{equation*}

From the facts that $d_{0}=d(P)>h(P,u_{i_{0}})-\xi_{p}(P_{\delta})\cdot u_{i_{0}}>h(P,u_{i_{0}})-\xi_{p}(P_{\delta})\cdot u_{i_{0}}-\delta>0$, $0<p<1$ and the function $f(t)=t^{p}$ is concave on $[0,\infty)$, we have
$$
\big(h(P,u_{i_{0}})-\xi_{p}(P_{\delta})\cdot u_{i_{0}}-\delta\big)^{p}-\big(h(P,u_{i_{0}})-\xi_{p}(P_{\delta})\cdot
u_{i_{0}}\big)^{p}<(d_{0}-\delta)^{p}-d_{0}^{p}.
$$
Then,
\begin{equation}\tag{3.8}\label{Equation 3.8}
\begin{split}
B(\delta)&=(\lambda^{p}-1)\left(\sum_{k=1}^{N}\alpha_{k}\big(h(P,u_{k})-\xi_{p}(P_{\delta})\cdot u_{k}\big)^{p}\right)\\
&\qquad+\alpha_{i_{0}}\lambda^{p}\bigg[\big(h(P,u_{i_{0}})-\xi_{p}(P_{\delta})\cdot
u_{i_{0}}-\delta\big)^{p}-\big(h(P,u_{i_{0}})-\xi_{p}(P_{\delta})\cdot
u_{i_{0}}\big)^{p}\bigg]\\
&<\left[(1-(c_{n}\delta^{n}+...+c_{2}\delta^{2}))^{-\frac{p}{n}}-1\right]\left(\sum_{k=1}^{N}\alpha_{k}(h(P,u_{k})-\xi_{p}(P_{\delta})\cdot
u_{k})^{p}\right)\\
&\qquad+\alpha_{i_{0}}\lambda^{p}\big[(d_{0}-\delta)^{p}-d_{0}^{p}\big].
\end{split}
\end{equation}
On the other hand,
\begin{equation}\tag{3.9}\label{Equation 3.9}
\lim_{\delta\rightarrow0}\sum_{k=1}^{N}\alpha_{k}\big(h(P,u_{k})-\xi_{p}(P_{\delta})\cdot u_{k}\big)^{p}=\sum_{k=1}^{N}\alpha_{k}h(P,u_{k})^{p},
\end{equation}
\begin{equation}\tag{3.10}\label{Equation 3.10}
(d_{0}-\delta)^{p}-d_{0}^{p}<0,
\end{equation}
and
\begin{equation}\tag{3.11}\label{Equation 3.11}
\begin{split}
\lim_{\delta\rightarrow0}&\frac{(1-(c_{n}\delta^{n}+...+c_{2}\delta^{2}))^{-\frac{p}{n}}-1}{(d_{0}-\delta)^{p}-d_{0}^{p}}\\
&=\lim_{\delta\rightarrow0}\frac{(-\frac{p}{n})(1-(c_{n}\delta^{n}+...+c_{2}\delta^{2}))^{-\frac{p}{n}-1}(-nc_{n}\delta^{n-1}-...-2c_{2}\delta)}{p(d_{0}-\delta)^{p-1}(-1)}=0.
\end{split}
\end{equation}

From Equations (\ref{Equation 3.8}), (\ref{Equation 3.9}), (\ref{Equation 3.10}), (\ref{Equation 3.11}), and the fact that $0<p<1$, we have $B(\delta)<0$ for small enough $\delta>0$. From this and Equation (\ref{Equation 3.7}), there
exists a $\delta_{0}>0$ such that $P_{\delta_{0}}\in\mathcal{P}(u_{1},...,u_{N})$ and
$$
\Phi_{\lambda_{0}P_{\delta_{0}}}(\xi_{p}(\lambda_{0}P_{\delta_{0}}))<\Phi_{P}(\xi_{p}(P_{\delta_{0}}))\leq\Phi_{P}(\xi_{p}(P))=\Phi_{P}(o),
$$
where $\lambda_{0}=V(P_{\delta_{0}})^{-\frac{1}{n}}$. Let $P_{0}=\lambda_{0}P_{\delta_{0}}-\xi_{p}(\lambda_{0}P_{\delta_{0}})$, then
$P_{0}\in\mathcal{P}^{n}(u_{1},...,u_{N})$, $V(P_{0})=1$, $\xi_{p}(P_{0})=o$ and
\begin{equation}\tag{3.12}\label{Equation 3.12}
\Phi_{P_{0}}(o)<\Phi_{P}(o).
\end{equation}
This contradicts Equation (\ref{Equation 3.6}). Therefore, $P\in\mathcal{P}_{N}(u_{1},...,u_{N})$.
\end{proof}

\section{The $L_{p}$ Minkowski problem for polytopes ($0<p<1$)}

In this section, we prove the main theorem of this paper. We only need prove the following:

\begin{theorem}
If $p\in(0,1)$, $\alpha_{1},...,\alpha_{N}\in\mathbb{R}^{+}$, and the unit vectors $u_{1},...,u_{N}$ ($N\geq n+1$) are not concentrated on a closed hemisphere, then there exists a polytope $P_{0}$ such that
$$
S_{p}(P_{0},\cdot)=\sum_{k=1}^{N}\alpha_{k}\delta_{u_{k}}(\cdot).
$$
\end{theorem}
\begin{proof}
From Lemma \ref{Lemma 3.6}, there exists a polytope $P\in\mathcal{P}_{N}(u_{1},...,u_{N})$ with $\xi_{p}(P)=o$ and $V(P)=1$ such that
$$
\Phi_{P}(o)=\inf\{\max_{\xi\in Q}\Phi_{Q}(\xi): Q\in\mathcal{P}(u_{1},...,u_{N})\textmd{ and }V(Q)=1\},
$$
where $\Phi_{Q}(\xi)=\sum_{k=1}^{N}\alpha_{k}(h(Q,u_{k})-\xi\cdot u_{k})^{p}$.

For $\delta_{1},...,\delta_{N}\in\mathbb{R}$, choose $|t|$ small enough so that the polytope $P_{t}$ defined by
$$
P_{t}=\bigcap_{i=1}^{N}\left\{x: x\cdot u_{i}\leq h(P,u_{i})+t\delta_{i}\right\}
$$
has exactly $N$ facets. Then,
$$
V(P_{t})=V(P)+t\left(\sum_{i=1}^{N}\delta_{i}a_{i}\right)+o(t),
$$
where $a_{i}$ is the area of $F(P,u_{i})$. Thus,
$$
\lim_{t\rightarrow0}\frac{V(P_{t})-V(P)}{t}=\sum_{i=1}^{N}\delta_{i}a_{i}.
$$
Let $\lambda(t)=V(P_{t})^{-\frac{1}{n}}$, then $\lambda(t)P_{t}\in\mathcal{P}_{N}^{n}(u_{1},...,u_{N})$, $V(\lambda(t)P_{t})=1$ and
\begin{equation}\tag{4.1}\label{Equation 4.1}
\lambda'(0)=-\frac{1}{n}\sum_{i=1}^{N}\delta_{i}S_{i}.
\end{equation}

Let $\xi(t)=\xi_{p}(\lambda(t)P_{t})$, and
\begin{equation}\tag{4.2}\label{Equation 4.2}
\begin{split}
\Phi(t)&=\max_{\xi\in\lambda(t)P_{t}}\sum_{k=1}^{N}\alpha_{k}(\lambda(t)h(P_{t},u_{k})-\xi\cdot u_{k})^{p}\\
&=\sum_{k=1}^{N}\alpha_{k}(\lambda(t)h(P_{t},u_{k})-\xi(t)\cdot u_{k})^{p}.
\end{split}
\end{equation}

From Equation (\ref{Equation 4.2}) and the fact that $\xi(t)$ is an interior point of $\lambda(t)P_{t}$, we have
\begin{equation}\tag{4.3}\label{Equation 4.3}
\sum_{k=1}^{N}\alpha_{k}\frac{u_{k,i}}{[\lambda(t)h(P_{t},u_{k})-\xi(t)\cdot u_{k}]^{1-p}}=0,
\end{equation}
for $i=1,...,n,$ where $u_{k}=(u_{k,1},...,u_{k,n})^{T}$. As a special case when $t=0$,
\begin{equation*}
\sum_{k=1}^{N}\alpha_{k}\frac{u_{k,i}}{h(P,u_{k})^{1-p}}=0,
\end{equation*}
for $i=1,...,n.$ Therefore,
\begin{equation}\tag{4.4}\label{Equation 4.4}
\sum_{k=1}^{N}\alpha_{k}\frac{u_{k}}{h(P,u_{k})^{1-p}}=0.
\end{equation}

Let
$$
F_{i}(t,\xi_{1},...,\xi_{n})=\sum_{k=1}^{N}\alpha_{k}\frac{u_{k,i}}{[\lambda(t)h(P_{t},u_{k})-(\xi_{1}u_{k,1}+...+\xi_{n}u_{k,n})]^{1-p}}
$$
for $i=1,...,n.$ Then,
$$
\frac{\partial F_{i}}{\partial\xi_{j}}\bigg|_{(0,...,0)}=\sum_{k=1}^{N}\frac{(1-p)\alpha_{k}}{h(P,u_{k})^{2-p}}u_{k,i}u_{k,j}.
$$
Thus,
$$
\left(\frac{\partial F}{\partial\xi}\bigg|_{(0,...,0)}\right)_{n\times n}=\sum_{k=1}^{N}\frac{(1-p)\alpha_{k}}{h(P,u_{k})^{2-p}}u_{k}\cdot
u_{k}^{T},
$$
where $u_{k}u_{k}^{T}$ is an $n\times n$ matrix.

Since $u_{1},...,u_{N}$ are not contained in a closed hemisphere, $\mathbb{R}^{n}=\Span\{u_{1},...,u_{N}\}$. Thus, for any $x\in\mathbb{R}^{n}$ with $x\neq0$, there exists a $u_{i_{0}}\in\{u_{1},...,u_{N}\}$ such that
$u_{i_{0}}\cdot x\neq0$. Then,
\begin{equation*}
\begin{split}
x^{T}\cdot\left(\sum_{k=1}^{N}\frac{(1-p)\alpha_{k}}{h(P,u_{k})^{2-p}}u_{k}\cdot u_{k}^{T}\right)\cdot x&=\sum_{k=1}^{N}\frac{(1-p)\alpha_{k}}{h(P,u_{k})^{2-p}}(x\cdot u_{k})^{2}\\
&\geq\frac{(1-p)\alpha_{i_{0}}}{h(P,u_{i_{0}})^{2-p}}(x\cdot u_{i_{0}})^{2}>0.
\end{split}
\end{equation*}
Thus, $(\frac{\partial F}{\partial\xi}\big|_{(0,...,0)})$ is positive defined. From this, Equations (\ref{Equation 4.3}) and the inverse function theorem, we
have
$$
\xi'(0)=(\xi_{1}'(0),...,\xi_{n}'(0))
$$
exists.

From the fact that $\Phi(0)$ is an extreme value of $\Phi(t)$ (in Equation (\ref{Equation 4.2})), Equation (\ref{Equation 4.1}) and Equation (\ref{Equation 4.4}), we have
\begin{equation*}
\begin{split}
0&=\Phi'(0)/p\\
&=\sum_{k=1}^{N}\alpha_{k}h(P,u_{k})^{p-1}\left(\lambda'(0)h(P,u_{k})+\delta_{k}-\xi'(0)\cdot u_{k}\right)\\
&=\sum_{k=1}^{N}\alpha_{k}h(P,u_{k})^{p-1}\left[-\frac{1}{n}\left(\sum_{i=1}^{N}a_{i}\delta_{i}\right)h(P,u_{k})+\delta_{k}\right]-\xi'(0)\cdot\left[\sum_{k=1}^{N}
\alpha_{k}\frac{u_{k}}{h(P,u_{k})^{1-p}}\right]\\
&=\sum_{k=1}^{N}\alpha_{k}h(P,u_{k})^{p-1}\delta_{k}-\left(\sum_{i=1}^{N}a_{i}\delta_{i}\right)\frac{\sum_{k=1}^{N}\alpha_{k}h(P,u_{k})^{p}}{n}\\
&=\sum_{k=1}^{N}\left(\alpha_{k}h(P,u_{k})^{p-1}-\frac{\sum_{j=1}^{N}\alpha_{j}h(P,u_{j})^{p}}{n}a_{k}\right)\delta_{k}.
\end{split}
\end{equation*}
Since $\delta_{1},...,\delta_{N}$ are arbitrary,
$$
\frac{\sum_{j=1}^{N}\alpha_{j}h(P,u_{j})^{p}}{n}h(P,u_{k})^{1-p}a_{k}=\alpha_{k},
$$
for all $k=1,...,N$. Let
$$
P_{0}=\left(\frac{\sum_{j=1}^{N}\alpha_{j}h(P,u_{j})^{p}}{n}\right)^{\frac{1}{n-p}}P,
$$
we have
$$
S_{p}(P_{0},\cdot)=\sum_{k=1}^{N}\alpha_{k}\delta_{u_{k}}(\cdot).
$$
\end{proof}

\section{The $L_{p}$ Minkowski problem for polytopes ($p\geq1$ with $p\neq n$)}

In \cite{HLYZ2}, Hug, et al. established a necessary and sufficient condition of the existence of the solution of the discrete $L_{p}$ Minkowski problem for the case where $p>1$ with $p\neq n$. In this section, we prove it by a different method. Moreover, this proof also includes a new approach to the classical Minkowski problem.

Let $p\geq1$, $\alpha_{1},...,\alpha_{N}\in\mathbb{R}^{+}$, the unit vectors $u_{1},...,u_{N}$ ($N\geq n+1$) are not concentrated on a closed hemisphere (in addition $\sum_{i=1}^{N}\alpha_{i}u_{i}=0$ for the case where $p=1$), $P\in\mathcal{P}(u_{1},...,u_{N})$, and $o\in P$. Define
$$
\Psi(P)=\sum_{i=1}^{N}\alpha_{i}h(P,u_{i})^{p}.
$$
Consider the extreme problem
\begin{equation}\tag{5.0}\label{Equation 5.0}
\inf\left\{\Psi(Q): Q\in\mathcal{P}(u_{1},...,u_{N}), V(Q)=1, o\in Q\right\}.
\end{equation}

 In this section, we prove that there exists a polytope $P$ with $u_{1},...,u_{N}$ as its unit facet normal vectors and $o\in\Int(P)$, which is the solution of problem (\ref{Equation 5.0}). Moreover, we prove that a dilatation of $P$ is the solution of the corresponding discrete $L_{p}$ Minkowski problem.

The following lemma will be needed.

\begin{lemma}\label{Lemma 5.1}
 Suppose the unit vectors $u_{1},...,u_{N}$ ($N\geq n+1$) are not concentrated on a closed hemisphere, $P\in\mathcal{P}(u_{1},...,u_{N})$ and $o\in P$. If there exists an $i_{0}$ ($1\leq i_{0}\leq N$) such that $h(P,u_{i_{0}})=0$ and $|F(P,u_{i_{0}})|>0$, then there exists a $\delta_{0}>0$ so that when $0<\delta<\delta_{0}$ the polytope
$$
P_{\delta}=\bigg(\bigcap_{i\neq i_{0}}H^{-}(P,u_{i})\bigg)\bigcap\{x: x\cdot u_{i_{0}}\leq\delta\}\in\mathcal{P}(u_{1},...,u_{N})
$$
and
$$
V(P_{\delta})=V(P)+(c_{n}\delta^{n}+...+c_{2}\delta^{2}+c_{1}\delta),
$$
where $c_{n},...,c_{2}$ are constants that depend on $P$ and $u_{i_{0}}$, and $c_{1}\neq0$.
\end{lemma}
\begin{proof}
By condition
$$
P_{1}=\bigg(\bigcap_{i\neq i_{0}}H^{-}(P,u_{i})\bigg)\bigcap\{x: x\cdot u_{i_{0}}\leq1\}\in\mathcal{P}(u_{1},...,u_{N}).
$$

Thus, (e.g., \cite{GGR}, Proposition 3.1) for $\delta\in\mathbb{R}$,
$$
g(\delta)=V_{n-1}(P_{1}\cap\{x: x\cdot v_{0}=\delta\})
$$
is a piecewise polynomial function of degree at most $n-1$. By conditions, $g(0)\neq0$. Thus, there exists a $\delta_{0}>0$ and $c_{n-1}',...,c_{1}', c_{0}'$ (depend on $P$ and $u_{i_{0}}$) such that $c_{0}'\neq0$ and when $0\leq\delta<\delta_{0}$
$$
g(\delta)=c_{n-1}'\delta^{n-1}+...+c_{1}'\delta+c_{0}'.
$$
Therefore, when $0\leq\delta<\delta_{0}$,
\begin{equation*}
\begin{split}
V(P_{\delta})&=V(P)+\int_{0}^{\delta}g(t)dt\\
&=V(P)+(c_{n}\delta^{n}+...+c_{2}\delta^{2}+c_{1}\delta),
\end{split}
\end{equation*}
where $c_{n}=c_{n-1}'/n,...,c_{2}=c_{1}'/2, c_{1}=c_{0}'$ are constants that depend on $P$ and $u_{i_{0}}$, and $c_{1}\neq0$.
\end{proof}

The following two lemmas solve problem (\ref{Equation 5.0}).
\begin{lemma}\label{Lemma 5.2}
Let $p\geq1$, $\alpha_{1},...,\alpha_{N}\in\mathbb{R}^{+}$, and the unit vectors $u_{1},...,u_{N}$ ($N\geq n+1$) are not concentrated on a closed hemisphere (in addition, $\sum_{i=1}^{N}\alpha_{i}u_{i}=0$ if $p=1$). Then, there exists a $P\in\mathcal{P}(u_{1},...,u_{N})$ with $o\in\Int(P)$ such that
$$
\Psi(P)=\inf\left\{\Psi(Q): Q\in\mathcal{P}(u_{1},...,u_{N}), V(Q)=1, o\in Q\right\}.
$$
\end{lemma}
\begin{proof}
Choose a $P_{0}\in\mathcal{P}(u_{1},...,u_{N})$ with $V(P_{0})=1$ and $o\in P$, then
\begin{equation}\tag{5.1}\label{Equation 5.1}
\inf\left\{\Psi(Q): Q\in\mathcal{P}(u_{1},...,u_{N}), V(Q)=1, o\in Q\right\}\leq\Psi(P_{0}).
\end{equation}

Choose a sequence $P_{k}\in\mathcal{P}(u_{1},...,u_{N})$ with $V(P_{k})=1$ and $o\in P_{k}$ such that $\Psi(P_{k})$ converges to $$
\inf\left\{\Psi(Q): Q\in\mathcal{P}(u_{1},...,u_{N}), V(Q)=1, o\in Q\right\}.
$$
We claim that $P_{k}$ is bounded. Otherwise from Lemma \ref{Lemma 3.4}, $\Psi(P_{k})$ is not bounded from above. This contradicts Equation (\ref{Equation 5.1}). Therefore, $P_{k}$ is bounded. From the Blaschke section theorem, there exists a subsequence of $P_{k}$ that converges to a polytope $P$ such that $P\in\mathcal{P}(u_{1},...,u_{N})$, $V(P)=1$, $o\in P$ and
\begin{equation}\tag{5.2}\label{Equation 5.2}
\Psi(P)=\inf\left\{\Psi(Q): Q\in\mathcal{P}(u_{1},...,u_{N}), V(Q)=1, o\in Q\right\}.
\end{equation}

 From conditions, if $p=1$, $P$, $Q$ contain the origin, and $P=Q+x$ for some $x\in\mathbb{R}^{n}$, then $\Psi(P)=\Psi(Q)$. Thus, we can assume $o$ is an interior point of $P$ (in (\ref{Equation 5.2})).

We next prove that when $p>1$ the origin is an interior point of $P$. Otherwise, there exists an $i_{0}$ ($1\leq i_{0}\leq N$) such that $h(P,u_{i_{0}})=0$ and $|F(P,u_{i_{0}})|>0.$

By Lemma \ref{Lemma 5.1}, we can choose $\delta>0$ small enough so that the polytope
$$
P_{\delta}=\bigg(\bigcap_{i\neq i_{0}}H^{-}(P,u_{i})\bigg)\bigcap\{x: x\cdot u_{i_{0}}\leq\delta\}\in\mathcal{P}(u_{1},...,u_{N})
$$
and
$$
V(P_{\delta})=1+(c_{n}\delta^{n}+...+c_{2}\delta^{2}+c_{1}\delta),
$$
where $c_{n},...,c_{1}$ are constants that depend on $P$ and $u_{i_{0}}$, and $c_{1}\neq0$.

Let $\lambda=V(P_{\delta})^{-\frac{1}{n}}$. From the hypothesis $h(P,u_{i_{0}})=0$,
\begin{equation*}
\begin{split}
\Psi(\lambda P_{\delta})&=\sum_{i=1}^{N}\alpha_{i}h(\lambda P_{\delta},u_{i})^{p}\\
&=\lambda^{p}\left[\sum_{i=1}^{N}\alpha_{i}h(P_{\delta},u_{i})^{p}\right]\\
&=\lambda^{p}\left[\sum_{i=1}^{N}\alpha_{i}h(P,u_{i})^{p}\right]+\lambda^{p}\alpha_{i_{0}}\big(h(P,u_{i_{0}})+\delta\big)^{p}-\lambda^{p}\alpha_{i_{0}}h(P,u_{i_{0}})^{p}\\
&=\sum_{i=1}^{N}\alpha_{i}h(P,u_{i})^{p}+(\lambda^{p}-1)\sum_{i=1}^{N}\alpha_{i}h(P,u_{i})^{p}+\lambda^{p}\alpha_{i_{0}}\delta^{p}\\
&=\Psi(P)+B_{1}(\delta),
\end{split}
\end{equation*}
where
\begin{equation*}
\begin{split}
B_{1}(\delta)&=(\lambda^{p}-1)\sum_{i=1}^{N}\alpha_{i}h(P,u_{i})^{p}+\lambda^{p}\alpha_{i_{0}}\delta^{p}\\
&=\left[\big(1+(c_{n}\delta^{n}+...+c_{2}\delta^{2}+c_{1}\delta)\big)^{-\frac{p}{n}}-1\right]\sum_{i=1}^{N}\alpha_{i}h(P,u_{i})^{p}+\lambda^{p}\alpha_{i_{0}}\delta^{p}.
\end{split}
\end{equation*}

Since $c_{1}\neq0$ and
\begin{equation*}
\begin{split}
\lim_{\delta\rightarrow0}&\frac{\delta^{p}}{(1+c_{n}\delta^{n}+...+c^{2}\delta^{2}+c_{1}\delta)^{-\frac{p}{n}}-1}\\
&=\lim_{\delta\rightarrow0}\frac{p\delta^{p-1}}{(-\frac{p}{n})(1+c_{n}\delta^{n}+...+c^{2}\delta^{2}+c_{1}\delta)^{-\frac{p}{n}-1}(nc_{n}\delta^{n-1}+...+c_{1})}=0,
\end{split}
\end{equation*}
$B_{1}(\delta)<0$ for small enough positive $\delta$. Thus
$$
\Psi(\lambda P_{\delta})<\Psi(P)
$$
for small enough positive $\delta$.  This contradicts Equation (\ref{Equation 5.2}). Therefore, the origin is an interior point of $P$.
\end{proof}

\begin{lemma}\label{Lemma 5.3}
The minimizing polytope in Lemma \ref{Lemma 5.2} has N facets.
\end{lemma}
\begin{proof}
If the statement of the Lemma is not true, then there exists an $i_{0}\in\{1,...,N\}$ such that
$$
F(P,u_{i_{0}})
$$
is not a facet of $P$.

By Lemma \ref{Lemma 3.5}, we can choose $\delta>0$ small enough so that the polytope
$$
P_{\delta}=P\cap\{x: x\cdot u_{i_{0}}\leq h(P,u_{i_{0}})-\delta\}\in\mathcal{P}(u_{1},...,u_{N}).
$$
and
$$
V(P_{\delta})=1-(c_{n}\delta^{n}+...+c_{2}\delta^{2}),
$$
where $c_{n},...,c_{2}$ are constants that depend on $P$ and direction $u_{i_{0}}$.

Let
$$
\lambda=\lambda(\delta)=V(P_{\delta})^{-\frac{1}{n}}=\big(1-(c_{n}\delta^{n}+...+c_{2}\delta^{2})\big)^{-\frac{1}{n}},
$$
then
\begin{equation*}
\begin{split}
\Psi(\lambda P_{\delta})&=\sum_{i=1}^{N}\alpha_{i}h(\lambda P_{\delta},u_{i})^{p}\\
&=\lambda^{p}\sum_{i=1}^{N}\alpha_{i}h(P_{\delta},u_{i})^{p}\\
&=\sum_{i=1}^{N}\alpha_{i}h(P,u_{i})^{p}+(\lambda^{p}-1)\sum_{i=1}^{N}\alpha_{i}h(P,u_{k})^{p}+\alpha_{i_{0}}\lambda^{p}\left[\left(h(P,u_{i_{0}})-\delta\right)^{p}
-h(P,u_{i_{0}})^{p}\right]\\
&=\Psi(P)+B_{2}(\delta),
\end{split}
\end{equation*}
where,
\begin{equation*}
\begin{split}
B_{2}(\delta)&=(\lambda^{p}-1)\sum_{i=1}^{N}\alpha_{i}h(P,u_{i})^{p}+\alpha_{i_{0}}\lambda^{p}\left[\left(h(P,u_{i_{0}})-\delta\right)^{p}
-h(P,u_{i_{0}})^{p}\right]\\
&=\left[\left(1-(c_{n}\delta^{n}+...+c_{2}\delta^{2})\right)^{-\frac{p}{n}}-1\right]\left(\sum_{i=1}^{N}\alpha_{i}h(P,u_{i})^{p}\right)
+\alpha_{i_{0}}\lambda^{p}\left[(a_{0}-\delta)^{p}-a_{0}^{p}\right],
\end{split}
\end{equation*}
and $a_{0}=h(P,u_{i_{0}})$.

Since $(a_{0}-\delta)^{p}-a_{0}^{p}<0$ for small positive $\delta$ and
\begin{equation*}
\begin{split}
\lim_{\delta\rightarrow0}&\frac{\big(1-(c_{n}\delta^{n}+...+c_{2}\delta^{2})\big)^{-\frac{p}{n}}-1}{(a_{0}-\delta)^{p}-a_{0}^{p}}\\
&=\lim_{\delta\rightarrow0}\frac{-\frac{p}{n}\big(1-(c_{n}\delta^{n}+...+c_{2}\delta^{2})\big)^{-\frac{p}{n}-1}(-nc_{n}\delta^{n-1}-...-2c_{2}\delta)}{p(a_{0}-\delta)^{p-1}(-1)}=0,
\end{split}
\end{equation*}
there exists a $\delta_{0}>0$ such that $B_{2}(\delta)<0$ for all $0<\delta<\delta_{0}$. Thus,
$$
\Psi(\lambda P_{\delta})<\Psi(P)
$$
for all $0<\delta<\delta_{0}$. This contradicts Equation (\ref{Equation 5.2}). Therefore, $P\in\mathcal{P}_{N}(u_{1},...,u_{N})$.
\end{proof}

 Suppose $P$ is a polytope with $N$ facets whose outer unit normals are $u_{1},...,u_{N}$ and such that the facet with outer normal $u_{k}$ has area $a_{k}$. Obviously, for any $u\in S^{n-1}$
$$
\sum_{u_{k}\cdot u\geq0}(u_{k}\cdot u)a_{k}=-\sum_{u_{k}\cdot u\leq0}(u_{k}\cdot u)a_{k},
$$
and both equal to the ($n-1$)-dimensional volume of the projection of $P$ on $u^{\perp}$. Thus, for all $u\in S^{n-1}$
$$
\sum_{k=1}^{N}(u_{k}\cdot u)a_{k}=0.
$$

Therefore,
\begin{equation}\tag{5.3}\label{Equation 5.3}
\sum_{k=1}^{N}a_{k}u_{k}=0.
\end{equation}

Equation (\ref{Equation 5.3}) is a necessary condition for the existence of the solution of the discrete classical Minkowski problem. The following theorem shows that it is also the sufficient condition. Moreover, the following theorem is the necessary and sufficient conditions for the existence of the solution of the discrete $L_{p}$ Minkowski problem for the case where $p>1$ with $p\neq n$.

\begin{theorem}
If $p\geq1$ with $p\neq n$, $\alpha_{1},...,\alpha_{N}\in\mathbb{R}^{+}$, and the unit vectors $u_{1},...,u_{N}$ ($N\geq n+1$) are not concentrated on a closed hemisphere (in addition, $\sum_{i=1}^{N}\alpha_{i}u_{i}=0$ if $p=1$), then there exists a polytope $P_{0}$ such that
$$
S_{p}(P_{0},\cdot)=\sum_{k=1}^{N}\alpha_{k}\delta_{u_{k}}(\cdot).
$$
\end{theorem}
\begin{proof}
By Lemma \ref{Lemma 5.2} and Lemma \ref{Lemma 5.3}, there exists a polytope $P\in\mathcal{P}_{N}(u_{1},...,u_{N})$ with $o\in\Int(P)$ and $V(P)=1$ such that
$$
\Psi(P)=\inf\left\{\Psi(Q): Q\in\mathcal{P}(u_{1},...,u_{N}), V(Q)=1, o\in Q\right\}.
$$

For $\delta_{1},...,\delta_{N}\in\mathbb{R}$, choose $|t|$ small enough so that the polytope $P_{t}$ defined by
$$
P_{t}=\bigcap_{i=1}^{N}\left\{x: x\cdot u_{i}\leq h(P,u_{i})+t\delta_{i}\right\}
$$
has exactly $N$ facets. Then,
$$
V(P_{t})=V(P)+t\left(\sum_{i=1}^{N}\delta_{i}a_{i}\right)+o(t),
$$
where $a_{i}$ is the area of $F(P,u_{i})$. Thus,
$$
\lim_{t\rightarrow0}\frac{V(P_{t})-V(P)}{t}=\sum_{i=1}^{N}\delta_{i}a_{i}.
$$
Let $\lambda(t)=V(P_{t})^{-\frac{1}{n}}$, then $\lambda(t)P_{t}\in\mathcal{P}_{N}^{n}(u_{1},...,u_{N})$, $o\in\Int(\lambda(t)P_{t})$, $V(\lambda(t)P_{t})=1$ and
\begin{equation*}
\lambda'(0)=-\frac{1}{n}\sum_{i=1}^{N}\delta_{i}S_{i}.
\end{equation*}

Let
$$
\Psi(t)=\Psi(\lambda(t)P_{t})=\sum_{i=k}^{N}\alpha_{k}\big(\lambda(t)h(P_{t},u_{k})\big)^{p}.
$$

From the fact that $\Psi(0)$ is an extreme value of $\Psi(t)$, we have
\begin{equation*}
\begin{split}
0&=\Psi'(0)/p\\
&=\sum_{k=1}^{N}\alpha_{k}h(P,u_{k})^{p-1}\left(\lambda'(0)h(P,u_{k})+\delta_{k}\right)\\
&=\sum_{k=1}^{N}\alpha_{k}h(P,u_{k})^{p-1}\left[-\frac{1}{n}\left(\sum_{i=1}^{N}a_{i}\delta_{i}\right)h(P,u_{k})+\delta_{k}\right]\\
&=\sum_{k=1}^{N}\alpha_{k}h(P,u_{k})^{p-1}\delta_{k}-\left(\sum_{i=1}^{N}a_{i}\delta_{i}\right)\frac{\sum_{k=1}^{N}\alpha_{k}h(P,u_{k})^{p}}{n}\\
&=\sum_{k=1}^{N}\left(\alpha_{k}h(P,u_{k})^{p-1}-\frac{\sum_{j=1}^{N}\alpha_{j}h(P,u_{j})^{p}}{n}a_{k}\right)\delta_{k}.
\end{split}
\end{equation*}
Since $\delta_{1},...,\delta_{N}$ are arbitrary,
$$
\frac{\sum_{j=1}^{N}\alpha_{j}h(P,u_{j})^{p}}{n}h(P,u_{k})^{1-p}a_{k}=\alpha_{k},
$$
for all $k=1,...,N$. Let
$$
P_{0}=\left(\frac{\sum_{j=1}^{N}\alpha_{j}h(P,u_{j})^{p}}{n}\right)^{\frac{1}{n-p}}P,
$$
we have
$$
S_{p}(P_{0},\cdot)=\sum_{k=1}^{N}\alpha_{k}\delta_{u_{k}}(\cdot).
$$
\end{proof}

\end{document}